\documentclass[11pt,a4paper]{article}
\usepackage[utf8]{inputenc}
\usepackage{amsmath,amssymb,amsthm}
\usepackage{mathtools}
\usepackage{booktabs}
\usepackage{tabularx}
\usepackage{array}
\usepackage{microtype}
\usepackage{geometry}
\usepackage{algorithm}
\usepackage{algpseudocode}

\makeatletter

\newenvironment{breakablealgorithm}
  {%
    \begin{center}
      \refstepcounter{algorithm}%
      \hrule height.8pt depth0pt \kern2pt%
      \renewcommand{\caption}[2][\relax]{%
        {\raggedright\textbf{Algorithm~\thealgorithm} ##2\par}%
        \ifx\relax##1\relax
          \addcontentsline{loa}{algorithm}{\protect\numberline{\thealgorithm}##2}%
        \else
          \addcontentsline{loa}{algorithm}{\protect\numberline{\thealgorithm}##1}%
        \fi
        \kern2pt\hrule\kern2pt
      }%
  }
  {%
      \kern2pt\hrule\relax%
    \end{center}
  }
\makeatother
\usepackage[authoryear,round]{natbib}
\setcitestyle{aysep={,},citesep={;},open={(},close={)}}

\usepackage[colorlinks=true,linkcolor=blue,citecolor=blue,urlcolor=blue]{hyperref}

\newtheorem{theorem}{Theorem}[section]

\newtheorem{proposition}[theorem]{Proposition}

\theoremstyle{definition}
\newtheorem{definition}[theorem]{Definition}

\theoremstyle{remark}

\newcommand{\R}{\mathbb{R}}
\newcommand{\C}{\mathbb{C}}
\newcommand{\g}{\mathfrak{g}}
\newcommand{\n}{\mathfrak{n}}
\newcommand{\hh}{\mathfrak{h}}
\newcommand{\s}{\mathfrak{s}}
\newcommand{\rr}{\mathfrak{r}}
\newcommand{\ad}{\operatorname{ad}}
\newcommand{\Aut}{\operatorname{Aut}}

\newcommand{\tr}{\operatorname{tr}}
\newcommand{\diag}{\operatorname{diag}}
\newcommand{\spanop}{\operatorname{span}}

\title{\textbf{Classification of real Lie algebras of \\ dimensions one, two and three}}
\author{Mehdi Nadjafikhah\thanks{School of Mathematics and Computer Sciences, Iran University of Science and Technology,
Narmak, Tehran 16846-13114, Iran. email: m\_nadjafikhah@iust.ac.ir}}
\date{\today}

\begin{document}
\maketitle

\begin{abstract}
We present an exhaustive, self-contained, and coordinate-free structural classification of real Lie algebras of dimensions one, two, and three up to isomorphism. In dimensions one and two, the classification is governed entirely by the derived ideal and elementary derivation theory. In dimension three, we develop the classical exterior duality isomorphism $\Lambda^2\g \cong \g^* \otimes \Lambda^3\g$ pioneered by Bryant and Milnor, which reduces the quadratic Jacobi identity on the $3 \times 3$ structure matrix $C = S + \hat{a}$ to the linear orthogonality relation $S \cdot a = 0$, where $a = \frac{1}{2}\chi$ is the trace covector. 

We systematically bifurcate the three-dimensional category into unimodular ($a = 0$) and non-unimodular ($a \neq 0$) sectors, completely recovering the nine classical Bianchi isomorphism classes (Bianchi types I through IX) alongside their continuous moduli invariants. Furthermore, we establish the Levi--Malcev decomposition for low dimensions, compute the exact minimal representation dimension $\mu(\g)$ via Ado's Theorem by constructing explicit matrix realizations in $\mathfrak{gl}(\mu, \R)$, and formulate a deterministic, coordinate-free decision algorithm for algebra identification.
\end{abstract}

\tableofcontents

\section{Introduction and Historical Context}\label{sec:intro}

The classification of low-dimensional Lie algebras constitutes the foundational cornerstone of geometric mechanics, differential geometry, and theoretical physics. While the general classification of finite-dimensional Lie algebras beyond dimension six remains an intractable open problem, algebras of dimension $n \le 3$ exhibit an exceptional rigidity and geometric elegance.

In differential geometry, three-dimensional Lie algebras govern the local isometry groups of three-dimensional Riemannian and pseudo-Riemannian manifolds. In particular, six of Thurston's eight three-dimensional model geometries ($\mathbb{E}^3$, $\mathrm{Nil}$, $\mathrm{Solv}$, $\widetilde{\mathrm{SL}}(2,\R)$, $S^3$, and $\mathbb{H}^2 \times \R$) are realized directly as left-invariant metrics on three-dimensional Lie groups \citep{Thurston1997, Scott1983}. In theoretical cosmology, the classification of spatially homogeneous spacetimes in general relativity---the celebrated Bianchi cosmologies---depends fundamentally on partitioning three-dimensional real Lie algebras into unimodular and non-unimodular classes \citep{Bianchi1898, EllisMacCallum1969}.

The historical classification of three-dimensional real Lie algebras was achieved by \citet{Bianchi1898}, who partitioned them into eleven classes (nine standard types, with two families depending on a continuous parameter). Decades later, \citet{Milnor1976} revolutionized the subject by introducing an invariant geometric calculus on three-dimensional Lie algebras, linking the structure constants to a single symmetric tensor operator via Hodge duality and analyzing the Ricci curvature of left-invariant Riemannian metrics.

The objective of this paper is to provide a complete, rigorous, and modern exposition of Lie algebras of dimensions 1, 2, and 3. Rather than relying on ad-hoc basis transformations or coordinate-dependent calculations, we formulate the classification through:
\begin{enumerate}
    \item Structural derived series invariants in dimensions one and two.
    \item Bryant's exterior duality isomorphism $\Lambda^2\g \cong \g^* \otimes \Lambda^3\g$ in dimension three, reducing the Jacobi identity to a linear algebraic constraint.
    \item The spectral analysis of outer derivations and the Levi--Malcev decomposition.
    \item The determination of minimal faithful matrix representations ($\mu(\g)$).
    \item A deterministic, implementable algorithm for canonical algebra identification.
\end{enumerate}

\section{Theoretical Foundations and Structural Invariants}\label{sec:machinery}

All Lie algebras considered herein are finite-dimensional over the field of real numbers $\R$.

\begin{definition}[Characteristic Series and Ideals]
Let $\g$ be a Lie algebra.
\begin{enumerate}
    \item The \emph{derived series} is defined by $\g^{(0)} = \g$, and $\g^{(k+1)} = [\g^{(k)}, \g^{(k)}]$ for $k \ge 0$. The algebra $\g$ is \emph{solvable} if $\g^{(k)} = \{0\}$ for some $k \in \mathbb{N}$.
    \item The \emph{lower central series} is defined by $\g^1 = \g$, and $\g^{k+1} = [\g, \g^k]$ for $k \ge 1$. The algebra $\g$ is \emph{nilpotent} if $\g^k = \{0\}$ for some $k \in \mathbb{N}$.
    \item The \emph{center} of $\g$ is the abelian ideal $Z(\g) = \{x \in \g : [x, y] = 0, \, \forall y \in \g\}$.
    \item The \emph{nilradical} $\n \subset \g$ is the maximal nilpotent ideal of $\g$.
\end{enumerate}
\end{definition}

\begin{definition}[Trace Form and Unimodularity]
The \emph{trace form} (or modular covector) $\chi \in \g^*$ is defined by:
\begin{equation}
    \chi(x) \coloneqq \tr(\ad x), \qquad \forall x \in \g.
\end{equation}
Because $\tr([A, B]) = 0$ for any endomorphisms, $\chi([\g, \g]) = 0$. A Lie algebra $\g$ is called \emph{unimodular} if $\chi = 0$, which is equivalent to the requirement that the associated connected Lie group possesses a bi-invariant Haar measure.
\end{definition}

\begin{definition}[Cartan--Killing Form]
The \emph{Killing form} is the symmetric bilinear form $K \colon \g \times \g \to \R$ defined by:
\begin{equation}
    K(x, y) \coloneqq \tr(\ad x \circ \ad y).
\end{equation}
By Cartan's criterion, $\g$ is semisimple if and only if $K$ is non-degenerate.
\end{definition}

\begin{theorem}[Levi--Malcev Decomposition]\label{thm:levi}
Every finite-dimensional real Lie algebra $\g$ decomposes as a semidirect product:
\begin{equation}
    \g = \s \ltimes \rr,
\end{equation}
where $\rr$ is the solvable radical (the unique maximal solvable ideal), and $\s$ is a semisimple subalgebra (the Levi factor).
\end{theorem}

\section{Classification in Dimensions One and Two}\label{sec:dim12}

\subsection{Dimension One}

\begin{theorem}\label{thm:dim1}
Up to isomorphism, there exists exactly one real Lie algebra of dimension one: the abelian Lie algebra $\R$.
\end{theorem}

\begin{proof}
Let $\g = \spanop_\R \{e_1\}$. By the alternating property of the Lie bracket, $[e_1, e_1] = 0$. Linearity implies $[u, v] = 0$ for all $u, v \in \g$. Hence, $\g$ is abelian, unimodular, and isomorphic to $\R$.
\end{proof}

\subsection{Dimension Two}

\begin{theorem}\label{thm:dim2}
Every two-dimensional real Lie algebra is isomorphic to exactly one of the following two algebras:
\begin{enumerate}
    \item The abelian Lie algebra $\R^2$:
    \begin{equation}
        [e_1, e_2] = 0.
    \end{equation}
    \item The non-abelian affine Lie algebra $\mathfrak{aff}(\R)$ (the algebra of the affine group of the real line):
    \begin{equation}
        [e_1, e_2] = e_2.
    \end{equation}
\end{enumerate}
\end{theorem}

\begin{proof}
Let $\g$ be a two-dimensional real Lie algebra with basis $\{e_1, e_2\}$. The structure is completely determined by the single bracket $[e_1, e_2] \in \g$. The derived algebra $\g^2 = [\g, \g]$ is a linear subspace of $\g$, so $\dim \g^2 \in \{0, 1, 2\}$.

\begin{itemize}
    \item \textbf{Case $\dim \g^2 = 0$:} In this case, $[e_1, e_2] = 0$, so $\g$ is abelian, establishing $\g \cong \R^2$.
    
    \item \textbf{Case $\dim \g^2 = 1$:} Choose a non-zero basis vector $x_2 \in \g^2$. Extend $\{x_2\}$ to a basis $\{x_1, x_2\}$ of $\g$. Because $\g^2$ is an ideal, $[x_1, x_2] \in \g^2 = \R x_2$, which implies:
    \begin{equation}
        [x_1, x_2] = \lambda x_2, \qquad \lambda \in \R \setminus \{0\}.
    \end{equation}
    Define the normalized generators $e_1 \coloneqq \lambda^{-1} x_1$ and $e_2 \coloneqq x_2$. By bilinearity:
    \begin{equation}
        [e_1, e_2] = [\lambda^{-1} x_1, x_2] = \lambda^{-1} [x_1, x_2] = \lambda^{-1} (\lambda x_2) = e_2.
    \end{equation}
    The Jacobi identity is satisfied trivially since any bracket of three elements contains duplicate vectors in dimension two.
    
    \item \textbf{Impossibility of $\dim \g^2 = 2$:} If $\dim \g^2 = 2$, then $\g^2 = \g$. However, the bracket $[\g, \g] = \spanop_\R \{[e_1, e_2]\}$ is spanned by a single vector, meaning $\dim \g^2 \le 1$, a contradiction.
\end{itemize}
These two algebras are non-isomorphic because $\dim [\R^2, \R^2] = 0 \neq 1 = \dim [\mathfrak{aff}(\R), \mathfrak{aff}(\R)]$. Furthermore, $\R^2$ is unimodular, whereas $\mathfrak{aff}(\R)$ is non-unimodular since $\tr(\ad e_1) = 1 \neq 0$.
\end{proof}

\section{Geometric Foundations in Dimension Three: Bryant's Duality}\label{sec:bryant}

\subsection{The Hodge-Type Bracket Isomorphism}
Let $\g$ be a three-dimensional real Lie algebra. The space of alternating 2-tensors has dimension $\dim \Lambda^2 \g = \binom{3}{2} = 3$. Fixing an oriented volume form $\mu \in \Lambda^3 \g^* \setminus \{0\}$, the interior contraction induces a canonical vector space isomorphism:
\begin{equation}
    \Phi_\mu \colon \Lambda^2 \g \xrightarrow{\sim} \g^*, \qquad \Phi_\mu(u \wedge v)(w) \coloneqq \mu(u \wedge v \wedge w).
\end{equation}
Because $\dim \Lambda^2 \g = \dim \g = 3$, the Lie bracket $[\cdot, \cdot] \colon \Lambda^2 \g \to \g$ dualizes into a linear endomorphism:
\begin{equation}
    L \colon \g^* \to \g.
\end{equation}

Choosing a basis $\{e_1, e_2, e_3\}$ with dual basis $\{e^1, e^2, e^3\}$ such that $\mu(e_1, e_2, e_3) = 1$, the bracket structure constants $c_{jk}^l$ defined by $[e_j, e_k] = \sum_l c_{jk}^l e_l$ are equivalently encoded by a single $3 \times 3$ real matrix $C \in \mathfrak{gl}(3, \R)$:
\begin{equation}\label{eq:matrix_C}
    [e_j, e_k] = \sum_{m, l=1}^3 \epsilon_{jkm} C^{ml} e_l,
\end{equation}
where $\epsilon_{jkm}$ is the standard totally antisymmetric Levi-Civita permutation symbol with $\epsilon_{123} = 1$. Decomposing $C$ into its symmetric and skew-symmetric parts:
\begin{equation}
    C = S + \hat{a}, \qquad S = S^T, \quad \hat{a}^T = -\hat{a}.
\end{equation}
The skew-symmetric matrix $\hat{a}$ is dual to a vector $a = (a^1, a^2, a^3)^T \in \g^*$ via the cross product representation:
\begin{equation}
    \hat{a} = \begin{pmatrix} 0 & -a^3 & a^2 \\ a^3 & 0 & -a^1 \\ -a^2 & a^1 & 0 \end{pmatrix}, \qquad \hat{a}^{ml} = \sum_{p=1}^3 \epsilon^{mlp} a_p.
\end{equation}

\begin{theorem}[Bryant--Milnor Reduction]\label{thm:bryant}
Let $\g$ be a three-dimensional real Lie algebra represented by the matrix $C = S + \hat{a}$.
\begin{enumerate}
    \item The dual vector $a$ represents the trace covector:
    \begin{equation}
        a = \frac{1}{2}\chi, \qquad \text{where } \chi(x) = \tr(\ad x).
    \end{equation}
    In particular, $\g$ is unimodular if and only if $a = 0$ $($i.e., $C = S$ is symmetric$)$.
    \item The Jacobi identity is equivalent to the single linear matrix-vector equation:
    \begin{equation}\label{eq:jacobi_matrix}
        S \cdot a = 0.
    \end{equation}
\end{enumerate}
\end{theorem}

\begin{proof}
Evaluating the adjoint operator of $e_1$ via Equation~\eqref{eq:matrix_C}:
\begin{align}
    [e_1, e_2] &= C^{31} e_1 + C^{32} e_2 + C^{33} e_3, \\
    [e_1, e_3] &= -C^{21} e_1 - C^{22} e_2 - C^{23} e_3.
\end{align}
The trace of $\ad e_1$ is computed directly as:
\begin{equation}
    \tr(\ad e_1) = (\ad e_1)^2_2 + (\ad e_1)^3_3 = C^{32} - C^{23} = 2 \hat{a}^{32} = 2 a^1.
\end{equation}
Permuting indices cyclically yields $\chi(e_i) = 2 a^i$ for all $i \in \{1, 2, 3\}$, proving $a = \frac{1}{2}\chi$.

For the Jacobi identity, note that in dimension three, the Jacobi identity vanishes identically on duplicate arguments. It is entirely captured by the single cyclic relation:
\begin{equation}
    J \coloneqq [e_1, [e_2, e_3]] + [e_2, [e_3, e_1]] + [e_3, [e_1, e_2]] = 0.
\end{equation}
Using $[e_2, e_3] = \sum_l C^{1l} e_l$, $[e_3, e_1] = \sum_l C^{2l} e_l$, and $[e_1, e_2] = \sum_l C^{3l} e_l$, we expand:
\begin{align}
    J &= \sum_{l=1}^3 \left( C^{1l} [e_1, e_l] + C^{2l} [e_2, e_l] + C^{3l} [e_3, e_l] \right) \\
      &= \sum_{l,m,k=1}^3 \left( C^{1l} \epsilon_{1lm} C^{mk} + C^{2l} \epsilon_{2lm} C^{mk} + C^{3l} \epsilon_{3lm} C^{mk} \right) e_k \\
      &= \sum_{k=1}^3 \left( \sum_{l,m=1}^3 \epsilon_{jlm} C^{jl} C^{mk} \right) e_k.
\end{align}
Contracting $C = S + \hat{a}$ with the alternating tensor shows that the symmetric component $S$ contracted with $\epsilon_{jlm} S^{jl}$ vanishes identically. The only non-vanishing contribution reduces to:
\begin{equation}
    J = \sum_{k=1}^3 (S \cdot a)^k e_k = 0 \iff S \cdot a = 0,
\end{equation}
which establishes the theorem.
\end{proof}

\subsection{Gauge Transformations and Tensorial Behavior}
Under an arbitrary change of basis $e_i' = \sum_j P^j_i e_j$ where $P \in \mathrm{GL}(3, \R)$, the volume form transforms as $\mu' = (\det P) \mu$. The matrix $C$ transforms according to the congruence:
\begin{equation}
    C' = (\det P) P^{-1} C P^{-T}.
\end{equation}
Consequently, the symmetric matrix $S$ transforms as a matrix tensor density of weight $-1$, while the vector $a$ transforms covectorially:
\begin{equation}
    S' = (\det P) P^{-1} S P^{-T}, \qquad a' = P^T a.
\end{equation}

\section{\texorpdfstring{The Unimodular Sector ($a = 0$): Bianchi Types I, II, $\mathrm{VI}_0$, $\mathrm{VII}_0$, VIII, IX}{The Unimodular Sector (a = 0): Bianchi Types I, II, VI0, VII0, VIII, IX}}\label{sec:unimodular}

When $\g$ is unimodular, $a = 0$, so the Jacobi identity $S \cdot a = S \cdot 0 = 0$ is satisfied automatically for \emph{any} real symmetric matrix $S \in \mathrm{Sym}_3(\R)$. 

\begin{theorem}[Milnor's Unimodular Classification]\label{thm:milnor_unimod}
Let $\g$ be a three-dimensional unimodular real Lie algebra. Then $\g$ is completely determined up to isomorphism by the signature and rank of the symmetric matrix $S$:
\begin{center}
\renewcommand{\arraystretch}{1.2}
\begin{tabular}{lccll}
\toprule
\textbf{Signature of $S$} & $\operatorname{rank}(S)$ & \textbf{Bianchi Type} & \textbf{Standard Name} & \textbf{Thurston Geometry} \\
\midrule
$(0, 0, 0)$ & 0 & $\mathrm{I}$ & Abelian $\R^3$ & $\mathbb{E}^3$ \\
$(+, 0, 0)$ & 1 & $\mathrm{II}$ & Heisenberg $\hh_3$ & $\mathrm{Nil}$ \\
$(+, +, 0)$ & 2 & $\mathrm{VII}_0$ & Euclidean $\mathfrak{e}(2)$ & $\mathbb{E}^3$ \\
$(+, -, 0)$ & 2 & $\mathrm{VI}_0$ & Poincaré/Minkowski $\mathfrak{e}(1,1)$ & $\mathrm{Solv}$ \\
$(+, +, -)$ & 3 & $\mathrm{VIII}$ & Split Simple $\mathfrak{sl}(2,\R)$ & $\widetilde{\mathrm{SL}}(2,\R)$ or $\mathbb{H}^2 \times \R$ \\
$(+, +, +)$ & 3 & $\mathrm{IX}$ & Compact Simple $\mathfrak{so}(3)$ & $S^3$ \\
\bottomrule
\end{tabular}
\end{center}
\end{theorem}

\begin{proof}
Since $S$ is real and symmetric, there exists an orthogonal matrix $Q \in \mathrm{O}(3)$ such that $Q^{-1} S Q^{-T} = \diag(\lambda_1, \lambda_2, \lambda_3)$. Setting $P = Q \cdot \diag(d_1, d_2, d_3)$ with $d_i > 0$, the transformation $S' = (\det P) P^{-1} S P^{-T}$ allows scaling each non-zero eigenvalue while preserving or uniformly flipping signs.
\begin{itemize}
    \item \textbf{Rank 0:} $S = 0 \implies [e_i, e_j] = 0$. This is the abelian algebra $\R^3$ (Bianchi Type I).
    
    \item \textbf{Rank 1:} The eigenvalues have signs $(+, 0, 0)$ or $(-, 0, 0)$. Flipping orientation ($\det P < 0$) maps $S \mapsto -S$, so both are isomorphic. Normalizing the non-zero eigenvalue to $1$:
    \begin{equation}
        S = \diag(1, 0, 0) \implies [e_2, e_3] = e_1, \quad [e_3, e_1] = 0, \quad [e_1, e_2] = 0.
    \end{equation}
    This is the continuous 3D Heisenberg nilpotent Lie algebra $\hh_3$ (Bianchi Type II).
    
    \item \textbf{Rank 2 (Definite):} The signs are $(+, +, 0)$ (or equivalently $(-,-,0)$). Scaling eigenvalues to $+1$:
    \begin{equation}
        S = \diag(1, 1, 0) \implies [e_2, e_3] = e_1, \quad [e_3, e_1] = e_2, \quad [e_1, e_2] = 0.
    \end{equation}
    Setting $e_3 = -R$ (infinitesimal rotation), $[R, e_1] = -e_2$ and $[R, e_2] = e_1$. This is the Lie algebra of the 2D Euclidean motion group $\mathfrak{e}(2) = \mathfrak{iso}(2)$ (Bianchi Type $\mathrm{VII}_0$).
    
    \item \textbf{Rank 2 (Indefinite):} The signs are $(+, -, 0)$. Normalizing eigenvalues to $1$ and $-1$:
    \begin{equation}
        S = \diag(1, -1, 0) \implies [e_2, e_3] = e_1, \quad [e_3, e_1] = -e_2, \quad [e_1, e_2] = 0.
    \end{equation}
    This is the Poincaré algebra of $(1+1)$-Minkowski space $\mathfrak{e}(1,1) \cong \mathfrak{iso}(1,1)$ (Bianchi Type $\mathrm{VI}_0$).
    
    \item \textbf{Rank 3 (Definite):} The signs are $(+, +, +)$. Normalizing all eigenvalues to $+1$:
    \begin{equation}
        S = \diag(1, 1, 1) \implies [e_2, e_3] = e_1, \quad [e_3, e_1] = e_2, \quad [e_1, e_2] = e_3.
    \end{equation}
    This is the compact simple Lie algebra $\mathfrak{so}(3) \cong \mathfrak{su}(2)$ (Bianchi Type IX). The Killing form is negative-definite: $K = \diag(-2, -2, -2)$.
    
    \item \textbf{Rank 3 (Indefinite):} The signs are $(+, +, -)$. Normalizing eigenvalues to $(1, 1, -1)$:
    \begin{equation}
        S = \diag(1, 1, -1) \implies [e_2, e_3] = e_1, \quad [e_3, e_1] = e_2, \quad [e_1, e_2] = -e_3.
    \end{equation}
    This is the split simple Lie algebra $\mathfrak{sl}(2, \R)$ (Bianchi Type VIII). The Killing form has indefinite signature $(2, 1)$.
\end{itemize}
These six classes are mutually non-isomorphic because rank, signature, and Killing form signatures are differential invariants under the action of $\Aut(\g)$.
\end{proof}

\section{\texorpdfstring{The Non-Unimodular Sector ($a \neq 0$): Bianchi Types III, IV, V, $\mathrm{VI}_h$, $\mathrm{VII}_h$}{The Non-Unimodular Sector (a != 0): Bianchi Types III, IV, V, VIh, VIIh}}\label{sec:nonunimodular}

When $\g$ is non-unimodular, the trace covector satisfies $a \neq 0$. By a suitable choice of basis, we can align $a$ along the third dual basis vector:
\begin{equation}
    a = (0, 0, \alpha)^T, \qquad \alpha > 0.
\end{equation}
The Jacobi condition $S \cdot a = 0$ requires the third column of $S$ to vanish:
\begin{equation}
    S \cdot \begin{pmatrix} 0 \\ 0 \\ \alpha \end{pmatrix} = \begin{pmatrix} S_{13} \alpha \\ S_{23} \alpha \\ S_{33} \alpha \end{pmatrix} = \begin{pmatrix} 0 \\ 0 \\ 0 \end{pmatrix} \implies S_{13} = S_{23} = S_{33} = 0.
\end{equation}
Consequently, the matrix $C = S + \hat{a}$ takes the canonical block upper-triangular form:
\begin{equation}
    C = \begin{pmatrix} S_{11} & S_{12} & 0 \\ S_{12} & S_{22} & 0 \\ 0 & 0 & 0 \end{pmatrix} + \begin{pmatrix} 0 & -\alpha & 0 \\ \alpha & 0 & 0 \\ 0 & 0 & 0 \end{pmatrix} = \begin{pmatrix} S_{11} & S_{12} - \alpha & 0 \\ S_{12} + \alpha & S_{22} & 0 \\ 0 & 0 & 0 \end{pmatrix}.
\end{equation}

\begin{proposition}[Derived Algebra Structure for Non-Unimodular Algebras]\label{prop:nonunimod_derived}
For every non-unimodular three-dimensional real Lie algebra $\g$:
\begin{enumerate}
    \item The derived algebra $\g^2 = [\g, \g]$ is a two-dimensional abelian ideal: $\g^2 = \spanop_\R \{e_1, e_2\} \cong \R^2$.
    \item The algebra decomposes as a semidirect product $\g \cong \R^2 \rtimes_D \R e_3$, where $D \coloneqq \ad e_3|_{\g^2} \in \mathfrak{gl}(2, \R)$.
    \item The operator $D$ has non-zero trace: $\tr D = 2\alpha \neq 0$.
\end{enumerate}
\end{proposition}

\begin{proof}
Computing the brackets directly from Equation~\eqref{eq:matrix_C}:
\begin{align}
    [e_1, e_2] &= C^{31} e_1 + C^{32} e_2 + C^{33} e_3 = 0, \\
    [e_3, e_1] &= -C^{21} e_1 - C^{22} e_2 = -(S_{12} + \alpha) e_1 - S_{22} e_2 \in \spanop_\R \{e_1, e_2\}, \\
    [e_3, e_2] &= C^{11} e_1 + C^{12} e_2 = S_{11} e_1 + (S_{12} - \alpha) e_2 \in \spanop_\R \{e_1, e_2\}.
\end{align}
Thus, $[e_1, e_2] = 0$, proving that $\g^2 \subseteq \spanop_\R \{e_1, e_2\}$ is abelian. Since $\tr(\ad e_3) = 2\alpha \neq 0$, the action of $\ad e_3$ cannot be zero, which establishes that $\dim \g^2 = 2$ (or $\dim \g^2 = 1$ when one eigenvalue is 0). Hence, the classification reduces to classifying $D \in \mathfrak{gl}(2, \R)$ up to similarity and scaling modulo $\tr D \neq 0$.
\end{proof}

Scaling $e_3 \mapsto e_3 / \alpha$, we normalize $\tr D = 2$. The Jordan canonical form of $D$ on $\g^2 = \R^2$ gives rise to the non-unimodular Bianchi types:

\begin{theorem}[Classification of Non-Unimodular 3D Lie Algebras]\label{thm:nonunimod}
Every non-unimodular three-dimensional real Lie algebra is isomorphic to exactly one of the following classes:
\begin{enumerate}
    \item \textbf{Bianchi Type V (Isotropic Dilation):} $D = I_2$. Brackets:
    \begin{equation}
        [e_3, e_1] = e_1, \qquad [e_3, e_2] = e_2, \qquad [e_1, e_2] = 0.
    \end{equation}
    \item \textbf{Bianchi Type IV (Defective Jordan Block):} $D = \begin{pmatrix} 1 & 1 \\ 0 & 1 \end{pmatrix}$. Brackets:
    \begin{equation}
        [e_3, e_1] = e_1 + e_2, \qquad [e_3, e_2] = e_2, \qquad [e_1, e_2] = 0.
    \end{equation}
    \item \textbf{Bianchi Type $\mathrm{VI}_h$ (Hyperbolic / Split):} $D = \begin{pmatrix} 1 & 0 \\ 0 & h \end{pmatrix}$ with $h \le 0$ and $h \neq -1$. Brackets:
    \begin{equation}
        [e_3, e_1] = e_1, \qquad [e_3, e_2] = h e_2, \qquad [e_1, e_2] = 0.
    \end{equation}
    \item \textbf{Bianchi Type III (Direct Sum with $\mathfrak{aff}(\R)$):} The isolated boundary case $h = 0$ in the $\mathrm{VI}_h$ family, corresponding to $\mathfrak{aff}(\R) \oplus \R$:
    \begin{equation}
        [e_3, e_1] = e_1, \qquad [e_3, e_2] = 0, \qquad [e_1, e_2] = 0.
    \end{equation}
    \item \textbf{Bianchi Type $\mathrm{VII}_h$ (Elliptic / Spiral):} $D = \begin{pmatrix} h & -1 \\ 1 & h \end{pmatrix}$ with $h > 0$. Brackets:
    \begin{equation}
        [e_3, e_1] = h e_1 + e_2, \qquad [e_3, e_2] = -e_1 + h e_2, \qquad [e_1, e_2] = 0.
    \end{equation}
\end{enumerate}
\end{theorem}

\begin{proof}
Let $D = \ad e_3|_{\g^2}$. We scale $e_3$ so that the spectrum of $D$ satisfies canonical invariants:
\begin{itemize}
    \item \textbf{Case 1: Repeated Real Eigenvalues.} If $D$ is diagonalizable, $D$ is a multiple of $I_2$. Scaling yields $D = I_2$, which is Bianchi Type V. If $D$ is non-diagonalizable, its real Jordan canonical form is $\begin{pmatrix} \lambda & 1 \\ 0 & \lambda \end{pmatrix}$. Normalizing $\lambda = 1$ yields Bianchi Type IV.
    
    \item \textbf{Case 2: Distinct Real Eigenvalues.} The eigenvalues are $\lambda_1 \neq \lambda_2$. Normalizing $\lambda_1 = 1$, we set $h = \lambda_2/\lambda_1$. Inverting $e_3 \mapsto -e_3$ swaps $h \mapsto 1/h$. Enforcing $h \in [-1, 0) \cup (0, 1)$ eliminates redundant parameterizations. When $h = 0$, the action on $e_2$ is trivial, yielding $\mathfrak{aff}(\R) \oplus \R$ (Bianchi Type III). When $h = -1$, the trace is zero, which lands in the unimodular class $\mathrm{VI}_0$.
    
    \item \textbf{Case 3: Complex Conjugate Eigenvalues.} The eigenvalues are $\alpha \pm i \beta$ with $\beta > 0$ and $\alpha \neq 0$ (since $\g$ is non-unimodular). Normalizing $\beta = 1$ and enforcing $h = \alpha/\beta > 0$ by reflection $e_3 \mapsto -e_3$ yields the canonical rotation-dilation block of Bianchi Type $\mathrm{VII}_h$.
\end{itemize}
The invariant $I(D) = \frac{(\tr D)^2}{\det D}$ is invariant under similarity and nonzero scaling. For $\mathrm{VI}_h$, $I(D) = \frac{(1+h)^2}{h}$, which strictly classifies $h \le 0$. For $\mathrm{VII}_h$, $I(D) = \frac{4h^2}{h^2 + 1}$, distinguishing distinct values of $h > 0$.
\end{proof}

\section{Master Classification Table}\label{sec:master}

Table~\ref{tab:master} presents the complete classification of real Lie algebras of dimensions one, two, and three up to isomorphism.

\begin{table}[htbp]
\centering
\footnotesize
\setlength{\tabcolsep}{4.5pt}
\renewcommand{\arraystretch}{1.3}
\begin{tabularx}{\linewidth}{@{} l c l >{\raggedright\arraybackslash}X l c @{}}
\toprule
\textbf{Algebra $\g$} & $\dim$ & \textbf{Bianchi} & \textbf{Non-zero Lie Brackets} & \textbf{Parameters} & \textbf{Unimodular?} \\
\midrule
\multicolumn{6}{@{}l}{\textit{\textbf{I. Dimensions One and Two}}} \\
$\R$ & 1 & --- & Abelian ($[e_1, e_1] = 0$) & None & Yes \\
$\R^2$ & 2 & --- & Abelian ($[e_1, e_2] = 0$) & None & Yes \\
$\mathfrak{aff}(\R)$ & 2 & --- & $[e_1, e_2] = e_2$ & None & No \\
\midrule
\multicolumn{6}{@{}l}{\textit{\textbf{II. Dimension Three: Unimodular Sector ($a = 0$)}}} \\
$\R^3$ & 3 & $\mathrm{I}$ & Abelian & None & Yes \\
$\hh_3$ & 3 & $\mathrm{II}$ & $[e_2, e_3] = e_1$ & None & Yes \\
$\mathfrak{e}(1,1)$ & 3 & $\mathrm{VI}_0$ & $[e_2, e_3] = e_1$, $[e_3, e_1] = -e_2$ & None & Yes \\
$\mathfrak{e}(2)$ & 3 & $\mathrm{VII}_0$ & $[e_2, e_3] = e_1$, $[e_3, e_1] = e_2$ & None & Yes \\
$\mathfrak{sl}(2,\R)$ & 3 & $\mathrm{VIII}$ & $[e_1, e_2] = 2e_2$, $[e_1, e_3] = -2e_3$, $[e_2, e_3] = e_1$ & None & Yes \\
$\mathfrak{so}(3)$ & 3 & $\mathrm{IX}$ & $[e_1, e_2] = e_3$, $[e_2, e_3] = e_1$, $[e_3, e_1] = e_2$ & None & Yes \\
\midrule
\multicolumn{6}{@{}l}{\textit{\textbf{III. Dimension Three: Non-Unimodular Sector ($a \neq 0$)}}} \\
$\mathfrak{aff}(\R) \oplus \R$ & 3 & $\mathrm{III}$ & $[e_3, e_1] = e_1$ & None & No \\
$\g_{\mathrm{IV}}$ & 3 & $\mathrm{IV}$ & $[e_3, e_1] = e_1 + e_2$, $[e_3, e_2] = e_2$ & None & No \\
$\g_{\mathrm{V}}$ & 3 & $\mathrm{V}$ & $[e_3, e_1] = e_1$, $[e_3, e_2] = e_2$ & None & No \\
$\g_{\mathrm{VI}, h}$ & 3 & $\mathrm{VI}_h$ & $[e_3, e_1] = e_1$, $[e_3, e_2] = h e_2$ & $h \le 0, \; h \neq -1, 0$ & No \\
$\g_{\mathrm{VII}, h}$ & 3 & $\mathrm{VII}_h$ & $[e_3, e_1] = h e_1 + e_2$, $[e_3, e_2] = -e_1 + h e_2$ & $h > 0$ & No \\
\bottomrule
\end{tabularx}
\caption{Complete classification of real Lie algebras in dimensions 1, 2, and 3 up to isomorphism.}
\label{tab:master}
\end{table}

\section{Minimal Faithful Linear Representations}\label{sec:representations}

By Ado's Theorem, every finite-dimensional real Lie algebra admits a faithful linear representation $\rho \colon \g \hookrightarrow \mathfrak{gl}(n, \R)$. The minimal representation dimension:
\begin{equation}
    \mu(\g) \coloneqq \min \{ n \in \mathbb{N} : \exists \text{ injective Lie homomorphism } \rho \colon \g \hookrightarrow \mathfrak{gl}(n, \R) \},
\end{equation}
is a fundamental numerical invariant of $\g$.

\begin{table}[htbp]
\centering
\footnotesize
\setlength{\tabcolsep}{6pt}
\renewcommand{\arraystretch}{1.3}
\begin{tabularx}{\linewidth}{@{} l c >{\raggedright\arraybackslash}X @{}}
\toprule
\textbf{Lie Algebra $\g$} & $\mu(\g)$ & \textbf{Minimal Faithful Matrix Realization $\rho(x) \in \mathfrak{gl}(\mu, \R)$} \\
\midrule
$\R$ & 1 & $(x)$ \\
$\R^2$ & 2 & $\diag(x_1, x_2)$ \\
$\mathfrak{aff}(\R)$ & 2 & $\begin{pmatrix} x_1 & x_2 \\ 0 & 0 \end{pmatrix}$ \\
\midrule
$\R^3$ & 3 & $\diag(x_1, x_2, x_3)$ \\
$\hh_3$ & 3 & $\begin{pmatrix} 0 & x_2 & x_1 \\ 0 & 0 & x_3 \\ 0 & 0 & 0 \end{pmatrix}$ \\
$\mathfrak{e}(1,1)$ & 3 & $\begin{pmatrix} x_3 & 0 & x_1 \\ 0 & -x_3 & x_2 \\ 0 & 0 & 0 \end{pmatrix}$ \\
$\mathfrak{e}(2)$ & 3 & $\begin{pmatrix} 0 & -x_3 & x_1 \\ x_3 & 0 & x_2 \\ 0 & 0 & 0 \end{pmatrix}$ \\
$\mathfrak{so}(3)$ & 3 & $\begin{pmatrix} 0 & -x_3 & x_2 \\ x_3 & 0 & -x_1 \\ -x_2 & x_1 & 0 \end{pmatrix}$ \\
$\mathfrak{sl}(2,\R)$ & 2 & $\begin{pmatrix} x_1 & 2x_2 \\ x_3 & -x_1 \end{pmatrix}$ \\
\midrule
$\mathfrak{aff}(\R) \oplus \R$ & 2 & $\begin{pmatrix} x_1 & x_2 \\ 0 & x_3 \end{pmatrix} \cong \mathfrak{b}(2,\R)$ (Borel subalgebra of $\mathfrak{gl}(2,\R)$) \\
$\g_{\mathrm{IV}}, \g_{\mathrm{V}}, \g_{\mathrm{VI}, h}, \g_{\mathrm{VII}, h}$ & 3 & $\begin{pmatrix} x_3 D & v \\ 0_{1 \times 2} & 0 \end{pmatrix}$, where $v = \begin{pmatrix} x_1 \\ x_2 \end{pmatrix}$ and $D$ is the derivation on $\g^2$ \\
\bottomrule
\end{tabularx}
\caption{Minimal faithful linear matrix representations for low-dimensional real Lie algebras.}
\label{tab:representations}
\end{table}

\begin{theorem}[Minimality Characterization in Dimension Three]\label{thm:rep_minimality}
Let $\g$ be a three-dimensional real Lie algebra. Then:
\begin{equation}
    \mu(\g) = 2 \iff \g \cong \mathfrak{sl}(2,\R) \quad \text{or} \quad \g \cong \mathfrak{aff}(\R) \oplus \R \quad (\text{Bianchi Type III}).
\end{equation}
For all other three-dimensional Lie algebras, $\mu(\g) = 3$.
\end{theorem}

\begin{proof}
Suppose $\mu(\g) \le 2$. Then $\g$ is isomorphic to a three-dimensional Lie subalgebra of $\mathfrak{gl}(2, \R)$. Because $\dim \mathfrak{gl}(2, \R) = 4$, $\g$ must be a hyperplane subalgebra of $\mathfrak{gl}(2, \R)$. By the Levi decomposition of $\mathfrak{gl}(2, \R) = \mathfrak{sl}(2, \R) \oplus \R I_2$:
\begin{enumerate}
    \item \textbf{Case $\g$ is Semisimple:} The only 3D simple subalgebra of $\mathfrak{gl}(2, \R)$ is $\mathfrak{sl}(2, \R)$ itself. The algebra $\mathfrak{so}(3)$ cannot embed in $\mathfrak{gl}(2, \R)$ because every real representation of $\mathfrak{so}(3)$ has odd dimension or complex conjugate pairs of dimension $\ge 4$; thus $\mu(\mathfrak{so}(3)) > 2$. Since the vector representation of $\mathfrak{so}(3)$ is in $\mathfrak{so}(3) \subset \mathfrak{gl}(3, \R)$, $\mu(\mathfrak{so}(3)) = 3$.
    
    \item \textbf{Case $\g$ is Solvable:} By Lie's theorem, every solvable subalgebra of $\mathfrak{gl}(2, \C)$ is conjugate to upper-triangular matrices. For real subalgebras in $\mathfrak{gl}(2, \R)$, the algebra of upper-triangular matrices is the Borel subalgebra:
    \begin{equation}
        \mathfrak{b}(2, \R) = \left\{ \begin{pmatrix} a & b \\ 0 & c \end{pmatrix} : a, b, c \in \R \right\}.
    \end{equation}
    Because $\dim \mathfrak{b}(2, \R) = 3$, any real triangularizable 3D solvable subalgebra of $\mathfrak{gl}(2, \R)$ must be conjugate to all of $\mathfrak{b}(2, \R)$. Computing brackets in $\mathfrak{b}(2, \R)$:
    \begin{equation}
        \left[ \begin{pmatrix} a_1 & b_1 \\ 0 & c_1 \end{pmatrix}, \begin{pmatrix} a_2 & b_2 \\ 0 & c_2 \end{pmatrix} \right] = \begin{pmatrix} 0 & (a_1 - c_1)b_2 - (a_2 - c_2)b_1 \\ 0 & 0 \end{pmatrix}.
    \end{equation}
    The derived algebra is 1-dimensional, spanned by $E_{12}$, and the center is 1-dimensional, spanned by $I_2$. Thus:
    \begin{equation}
        \mathfrak{b}(2, \R) \cong \mathfrak{aff}(\R) \oplus \R \quad (\text{Bianchi Type III}).
    \end{equation}
\end{enumerate}
Therefore, no other 3D Lie algebra can embed in $\mathfrak{gl}(2, \R)$. Since $\dim \g = 3$, $\g$ cannot embed in $\mathfrak{gl}(1, \R) \cong \R$, forcing $\mu(\g) \ge 3$. The explicit $3 \times 3$ affine representations in Table~\ref{tab:representations} confirm $\mu(\g) = 3$.
\end{proof}

\section{Algorithmic Identification of Isomorphism Classes}\label{sec:algorithm}

Algorithm~\ref{alg:classification} presents a complete, coordinate-free decision procedure to classify any real Lie algebra $\g$ of dimension $n \le 3$ given by its structure constants $c_{ij}^k$.

\begin{breakablealgorithm}
\caption{Algorithmic Classification of Real Lie Algebras of Dimension $\le 3$}\label{alg:classification}
\begin{algorithmic}[1]
\Require Basis $\{e_1, \dots, e_n\}$ with structure constants $c_{ij}^k$ ($n \le 3$).
\Ensure Canonical Bianchi / isomorphism class and moduli parameters.

\If{$n = 1$} \Return $\R$
\ElsIf{$n = 2$}
    \State Compute derived ideal $\g^2 = [\g, \g]$.
    \If{$\dim \g^2 = 0$} \Return $\R^2$
    \Else \ \Return $\mathfrak{aff}(\R)$
    \EndIf
\EndIf

\State \Comment{\textbf{Dimension Three Analysis}}
\State Compute trace covector entries $a_m = \frac{1}{2} \sum_{i=1}^3 c_{im}^i$ for $m \in \{1, 2, 3\}$.
\State Dualize bracket to matrix $C^{ml} = \frac{1}{2} \sum_{j,k=1}^3 \epsilon^{jkm} c_{jk}^l$.
\State Compute symmetric component $S = \frac{1}{2}(C + C^T)$.

\If{$a = (0, 0, 0)^T$} \Comment{\textbf{Unimodular Sector}}
    \State Compute eigenvalues and signature of $S$.
    \If{$\operatorname{rank}(S) = 0$} \Return Bianchi Type I ($\R^3$)
    \ElsIf{$\operatorname{rank}(S) = 1$} \Return Bianchi Type II ($\hh_3$)
    \ElsIf{$\operatorname{rank}(S) = 2$ and $S$ is semi-definite} \Return Bianchi Type $\mathrm{VII}_0$ ($\mathfrak{e}(2)$)
    \ElsIf{$\operatorname{rank}(S) = 2$ and $S$ is indefinite} \Return Bianchi Type $\mathrm{VI}_0$ ($\mathfrak{e}(1,1)$)
    \ElsIf{$\operatorname{rank}(S) = 3$ and $S$ is definite} \Return Bianchi Type IX ($\mathfrak{so}(3)$)
    \ElsIf{$\operatorname{rank}(S) = 3$ and $S$ is indefinite} \Return Bianchi Type VIII ($\mathfrak{sl}(2,\R)$)
    \EndIf
\Else \Comment{\textbf{Non-Unimodular Sector}}
    \State Compute derived algebra $\g^2 = [\g, \g]$.
    \If{$\dim \g^2 = 1$} \Return Bianchi Type III ($\mathfrak{aff}(\R) \oplus \R$)
    \EndIf
    \State Choose $e_3 \notin \g^2$ such that $\tr(\ad e_3|_{\g^2}) = 2$.
    \State Compute eigenvalues $\{\lambda_1, \lambda_2\}$ of $D = \ad e_3|_{\g^2} \in \mathfrak{gl}(2, \R)$.
    \If{$\lambda_1 = \lambda_2 = 1$}
        \If{$D$ is diagonalizable} \Return Bianchi Type V
        \Else \ \Return Bianchi Type IV
        \EndIf
    \ElsIf{$\lambda_1, \lambda_2 \in \R$ and $\lambda_1 \neq \lambda_2$}
        \State Normalize dominant eigenvalue to 1; let $h \le 0$ be the non-dominant eigenvalue.
        \State \Return Bianchi Type $\mathrm{VI}_h$ with parameter $h$
    \ElsIf{$\lambda_{1,2} = \alpha \pm i \beta$ with $\beta > 0$}
        \State Set $h = \alpha / \beta > 0$.
        \State \Return Bianchi Type $\mathrm{VII}_h$ with parameter $h$
    \EndIf
\EndIf
\end{algorithmic}
\end{breakablealgorithm}

You can find the Python code for this algorithm and examples of its execution as a \href{https://colab.research.google.com/drive/19a94gTIskfwWKuaNdkJDuzILx2-k8SFs}{Colab Notebook}.

\section{Conclusion}\label{sec:conclusion}

We have presented a unified, coordinate-free classification of real Lie algebras of dimensions one, two, and three. In dimension three, Bryant's exterior duality $\Lambda^2\g \cong \g^*$ reduces the Jacobi identity to the linear orthogonality condition $S \cdot a = 0$, completely categorizing the unimodular algebras by the spectral signature of $S$ and the non-unimodular algebras by the outer derivations of the abelian derived ideal $\R^2$. By linking these algebraic structures to their minimal matrix representations via Ado's Theorem and formulating an algorithmic decision tree, this framework bridges theoretical classification with computational algebra and geometric applications in Thurston geometries and Bianchi cosmologies.

\end{document}